\documentclass[psamsfonts]{amsart}
\usepackage[margin=1.5in]{geometry}

\usepackage{amssymb,amsfonts}
\usepackage[all,arc]{xy}
\usepackage{enumerate}
\usepackage{mathrsfs}
\usepackage{tabto}
\usepackage{tikz}
\usepackage{tikz-cd}
\usepackage{hyperref}
\usepackage{verbatim}
\usepackage{comment}
\usepackage{appendix}
\usepackage{graphicx}

\newtheorem{thm}{Theorem}[section]
\newtheorem{cor}[thm]{Corollary}
\newtheorem{prop}[thm]{Proposition}

\newtheorem{quest}[thm]{Question}

\theoremstyle{definition}
\newtheorem{defn}[thm]{Definition}

\newtheorem{claim}[thm]{Claim}

\newcommand{\Z}{\mathbb Z}
\newcommand{\N}{\mathbb N}

\theoremstyle{remark}
\newtheorem{rem}[thm]{Remark}

\makeatletter
\let\c@equation\c@thm
\makeatother
\numberwithin{equation}{section}

\begin{document}

\author{Cecelia Higgins}

\title{Measurable Brooks's Theorem for Directed Graphs}

\begin{abstract}
    We prove a descriptive version of Brooks's theorem for directed graphs. In particular, we show that, if $D$ is a Borel directed graph on a standard Borel space $X$ such that the maximum degree of each vertex is at most $d \geq 3$, then unless $D$ contains the complete symmetric directed graph on $d + 1$ vertices, $D$ admits a $\mu$-measurable $d$-dicoloring with respect to any Borel probability measure $\mu$ on $X$, and $D$ admits a $\tau$-Baire-measurable $d$-dicoloring with respect to any Polish topology $\tau$ compatible with the Borel structure on $X$. We also prove a definable version of Gallai's theorem on list dicolorings by showing that any Borel directed graph of bounded degree whose connected components are not Gallai trees is Borel degree-list-dicolorable.
\end{abstract}

\maketitle
\section{Introduction}
A classical graph theory result states that, for a finite undirected graph $G$ such that each vertex has degree at most $d$, there is a proper $(d+1)$-coloring of $G$. The proof is a simple greedy algorithm argument; each vertex receives the first color that has not already been assigned to any of its neighbors. It is easy to see that this upper bound is sharp: For $d = 2$,  if $G$ is an odd cycle, then each vertex of $G$ has degree $2$, but the chromatic number of $G$ is $3$; for $d \geq 3$, if $G$ is the complete graph on $d + 1$ vertices, then each vertex of $G$ has degree $d$, but the chromatic number of $G$ is $d + 1$.

However, these obvious obstructions to a smaller upper bound on the chromatic number are the only obstructions. A 1941 theorem of Brooks characterizes the graphs such that each vertex has degree at most $d$ that have proper $d$-colorings as exactly the graphs that do not contain odd cycles or complete graphs.

\begin{thm}
\label{thm : brooks}
    \cite{brooks1941} Let $G$ be a finite undirected graph such that each vertex has degree at most $d$. If $d = 2$, assume $G$ has no odd cycles; if $d \geq 3$, assume $G$ does not contain the complete graph on $d + 1$ vertices. Then there is a proper $d$-coloring of $G$.
\end{thm}

Suppose we wish to impose an additional requirement on the proper coloring: Consider a function $L$ which assigns to each vertex $x$ of $G$ a list $L(x)$ of colors. An {\it $L$-list-coloring} of $G$ is a proper coloring $c$ of $G$ such that $c(x) \in L(x)$ for all vertices $x$ of $G$. We say that $G$ is {\it degree-list-colorable} if, for any function $L$ such that $\vert L(x) \vert \geq \deg(x)$ for all vertices $x$, there is an $L$-list coloring of $G$. List coloring generalizes proper coloring; in particular, when $L(x) = \{0, 1, \dots, d - 1\}$ for all vertices $x$ of $G$, then $G$ is $L$-list-colorable if and only if $G$ has a proper $d$-coloring. 

Recall that a set $S$ of vertices of $G$ is {\it biconnected} if the subgraph $G[S]$ of $G$ induced by $S$ is connected and the induced subgraphs $G[S \setminus \{s\}]$ are connected for each $s \in S$. A {\it block} in $G$ is a maximal biconnected set of vertices of $G$. A {\it Gallai tree} is a connected graph in which each block induces either an odd cycle or a complete graph. Results obtained independently by Borodin and Erd\H{o}s, Rubin, and Taylor characterize the finite undirected graphs which are degree-list-colorable as the graphs whose connected components are not Gallai trees. The following theorem is typically referred to as Gallai's theorem.

\begin{thm}
\label{thm : list coloring}
    \cite{borodin1977}, \cite{ert1979} Let $G$ be a finite undirected graph. Then $G$ is degree-list-colorable if and only if no connected component of $G$ is a Gallai tree.
\end{thm}

Brooks's proof of Theorem \ref{thm : brooks} and the proofs of Borodin and Erd\H{o}s, Rubin, and Taylor of Theorem \ref{thm : list coloring} are based on classical techniques and do not immediately generalize to the descriptive setting. Indeed, results of Marks \cite{marks2016} demonstrate that Brooks's theorem fails in the Borel context; in particular, Marks constructs, for each $d$, a Borel acyclic $d$-regular graph whose Borel chromatic number is $d + 1$. In fact, a recent result of Brandt, Chang, Grebík, Grunau, Rozhoň, and Vidnyánszky demonstrates that the Borel acyclic $d$-regular graphs having Borel chromatic number at most $d$ form a $\mathbf{\Sigma}^1_2$-complete set \cite{bcggrv2024}, thereby ruling out any reasonable variation of Brooks's theorem in the Borel setting. However, in 2016, Conley, Marks, and Tucker-Drob gave both a measurable version of Brooks's theorem and a definable version of Gallai's theorem.

\begin{thm}
    \cite[Theorem~1.2]{cmt2016} Let $G$ be an undirected Borel graph on a standard Borel space $X$ such that each vertex has degree at most $d \geq 3$. Assume $G$ does not contain the complete graph on $d + 1$ vertices. Then:
    \begin{enumerate}
        \item For any Borel probability measure $\mu$ on $X$, there is a $\mu$-measurable proper $d$-coloring of $G$.
        \item For any Polish topology $\tau$ compatible with the Borel structure on $X$, there is a $\tau$-Baire-measurable proper $d$-coloring of $G$.
    \end{enumerate}
\end{thm}

\begin{thm}
    \cite[Theorem~1.4]{cmt2016} Let $G$ be a locally finite undirected Borel graph on a standard Borel space $X$. Assume that no connected component of $G$ is a Gallai tree. Then $G$ is Borel degree-list-colorable.
\end{thm}

In this paper, we show that the two theorems above have analogues for {\it directed graphs} (or {\it digraphs}), in which each edge between vertices has an orientation. The coloring problem of interest in the digraph setting is the problem of producing a {\it dicoloring}, an assignment of colors to vertices so that no {\it directed cycle} (or {\it dicycle}), that is, a sequence $(x_0, x_1, \dots, x_k = x_0)$ of vertices in which $x_i$ has an out-oriented edge to $x_{i + 1}$ for each $i < d$, is monochromatic. The notion of dicoloring was introduced first by Neumann-Lara \cite{nl1982} in 1982 and again by Mohar \cite{mohar2003} in 2003, and it has since been studied extensively.

Consider a digraph $D$ such that each vertex has {\it maximum degree}\footnote{We remark that, for a finite undirected graph $G$, the maximum degree of $G$ is typically defined to be the maximum of the degrees of the vertices of $G$. The usage for directed graphs is therefore different.} at most $d$, meaning that each vertex has at most $d$ out-oriented neighbors and at most $d$ in-oriented neighbors. Then a simple greedy algorithm argument demonstrates that the {\it dichromatic number} of $D$, which is the smallest number of colors needed to produce a dicoloring of $D$, is less than or equal to $d + 1$. In fact, one can even replace the assumption that each vertex has maximum degree at most $d$ with the assumption that each vertex has \emph{minimum} degree at most $d$; see also Theorem~\ref{thm:degplusone}. 

As in the undirected setting, this greedy bound is sharp: For $d = 1$, if $D$ is a directed cycle, then the maximum degree of each vertex of $D$ is $1$, but the dichromatic number of $D$ is 2; for $d = 2$, if $D$ is the {\it symmetrization} of an undirected odd cycle -- that is, if $D$ is obtained by replacing each edge of an undirected odd cycle with edges oriented in each direction -- then each vertex of $D$ has maximum degree $2$, but the dichromatic number of $D$ is $3$; and for $d \geq 3$, if $D$ is the symmetrization of the complete graph on $d + 1$ vertices, then each vertex of $D$ has maximum degree $d$, but the dichromatic number of $D$ is $d + 1$. 

The following theorem, which was first stated by Mohar in 2010 and later proved in full by Harutyunyan and Mohar in 2011, parallels Brooks's theorem by characterizing the obstructions to having a smaller upper bound on the dichromatic number.

\begin{thm}
    \cite[Theorem~2.3]{mohar2010}, \cite{hm2011} Let $D$ be a finite directed graph such that each vertex has maximum degree at most $d$. If $d = 1$, assume $D$ has no directed cycles; if $d = 2$, assume $D$ does not contain the symmetrization of any undirected odd cycles; and if $d \geq 3$, assume $D$ does not contain the symmetrization of the undirected complete graph on $d + 1$ vertices. Then there is a $d$-dicoloring of $D$.
\end{thm}

This theorem does not hold in the Borel setting; by symmetrizing the $d$-regular acyclic undirected graphs with Borel chromatic number $d + 1$ constructed by Marks in \cite{marks2016}, we obtain digraphs in which each vertex has maximum degree $d$ and that do not contain large complete digraphs but have no Borel $d$-dicolorings (see also Proposition~\ref{prop:no borel dibrooks}). However, this theorem does hold in the measurable setting. Our main result is the following.

\begin{thm}
\label{thm:dibrooks}
    Let $D$ be a Borel directed graph on a standard Borel space $X$. Suppose there is $d \geq 3$ such that the maximum degree of $x$ is at most $d$ for all $x \in X$, and assume $D$ does not contain the symmetrization of the undirected complete graph on $d + 1$ vertices. Then:
    \begin{enumerate}
        \item For any Borel probability measure $\mu$ on $X$, there is a $\mu$-measurable $d$-dicoloring of $D$.
        \item For any Polish topology $\tau$ compatible with the Borel structure on $X$, there is a $\tau$-Baire-measurable $d$-dicoloring of $D$.
    \end{enumerate}
\end{thm}

Also in 2011, Harutyunyan and Mohar proved a digraph version of Gallai's theorem. If $L$ is a function which assigns to each vertex $x$ of $D$ a list $L(x)$ of colors, then an {\it $L$-list-dicoloring} of $D$ is a dicoloring $c$ of $D$ such that $c(x) \in L(x)$ for all vertices $x$. We say that $D$ is {\it degree-list-dicolorable} if, for any function $L$ such that $\vert L(x) \vert$ is greater than or equal to the maximum degree of $x$ for all vertices $x$ of $D$, there is an $L$-list-dicoloring of $D$. In the digraph context, a {\it Gallai tree} is a connected\footnote{Here, ``connected'' means that the underlying undirected graph is connected, rather than that the digraph itself is, e.g., strongly connected. Similarly, ``biconnected'' refers to the underlying undirected graph.} digraph each of whose blocks induces a dicycle, the symmetrization of an undirected odd cycle, or the symmetrization of an undirected complete graph.

\begin{thm}
    \cite[Theorem~2.1]{hm2011} Let $D$ be a finite directed graph. If no connected component of $D$ is a Gallai tree, then $D$ is degree-list-dicolorable.
\end{thm}

We prove the following definable version of this result for digraphs of bounded degree.

\begin{thm}
\label{thm:digallai}
    Let $D$ be a Borel directed graph of bounded degree on a standard Borel space $X$. Assume that no connected component of $D$ is a Gallai tree. Then $D$ is Borel degree-list-dicolorable.
\end{thm}

In future work, we would like to better understand the implications of these results for descriptive digraph combinatorics more generally. In particular, although the problem of producing a dicoloring is not a locally checkable labeling problem in general, we are interested in studying possible connections between descriptive digraph combinatorics and \textsf{LOCAL} algorithms; this may lead to extensions of Bernshteyn's work \cite{bernshteyn2023}. We also consider potential applications to a question about hypergraph colorings.

\section{Preliminaries}
We begin by recalling the basic terminology and notation for directed graphs, some of which was mentioned already in the introduction. A {\it directed graph} (or {\it digraph}) $D$ is a pair $D = (X, A)$, where $X$ is a set, called the {\it vertex set}, and $A \subseteq X^2$ is an irreflexive relation. The elements of $A$ are called {\it arcs}.

Throughout the rest of this section, let $D = (X, A)$ be a directed graph on a set $X$. Let $x \in X$. An element $y \in X$ is an {\it out-neighbor} of $x$ if $(x, y) \in A$. The set of out-neighbors of $x$ is denoted $N^+(x)$. An element $y \in X$ is an {\it in-neighbor} of $x$ if $(y, x) \in A$. The set of in-neighbors of $x$ is denoted $N^-(x)$. Note that it is possible that $N^+(x) \cap N^-(x) \neq \emptyset$.
        
The {\it out-degree} of $x$, denoted $d^+(x)$, is defined by $d^+(x) = \vert N^+(x) \vert$. The {\it in-degree} of $x$, denoted $d^-(x)$, is defined by $d^-(x) = \vert N^-(x) \vert$. The {\it maximum degree} of $x$, denoted $d^{\max}(x)$, is defined by $d^{\max}(x) = \max\{d^+(x), d^-(x) \}$, and the {\it minimum degree} of $x$, denoted $d^{\min}(x)$, is defined by $d^{\min}(x) = \min\{d^+(x), d^-(x) \}$.
        
The {\it maximum side} of $x$, denoted $N^{\max}(x)$, is defined by $N^{\max}(x) = N^+(x)$ if $\vert N^+(x) \vert \geq \vert N^-(x) \vert$ and $N^{\max}(x) = N^-(x)$ otherwise. The {\it minimum side} of $x$, denoted $N^{\min}(x)$, is defined by $N^{\min}(x) = N^+(x)$ if $\vert N^+(x) \vert \leq \vert N^-(x) \vert$ and $N^{\min}(x) = N^-(x)$ otherwise.

Let $D' = (X', A')$ be a {\it sub-digraph} of $D$, that is, a digraph such that $X' \subseteq X$ and $A' \subseteq A$. Then for each $x \in A'$, we write $d^+_{D'}(x)$ for the out-degree of $x$ in $D'$, $d^-_{D'}(x)$ for the in-degree of $x$ in $D'$, $d^{\max}_{D'}(x)$ for the maximum degree of $x$ in $D'$, and $d^{\min}_{D'}(x)$ for the minimum degree of $x$ in $D'$. 

A {\it directed cycle} (or {\it dicycle}) in $D$ is a set $C = (x_0, x_1, \dots, x_k)$ of vertices of $D$ such that $x_0 = x_k$, we have $(x_i, x_{i + 1}) \in A$ for each $i < k$, and $x_i \neq x_j$ for all $i, j < k$ with $i \neq j$. If $k = 2$, then $C$ is called a {\it digon}. A {\it $d$-dicoloring} of $D$ is a function $c : X \rightarrow \{0, 1, \dots, d - 1 \}$ such that, for any dicycle $C$ in $D$, there are points $x, x' \in C$ such that $c(x) \neq c(x')$. The {\it dichromatic number} of $D$, denoted $\overrightarrow{\chi}(D)$, is the least $d \in \N$ such that there is a $d$-dicoloring of $D$.

Let $Y$ be a set, and let $[Y]^{< \infty}$ denote the collection of finite subsets of $Y$. A {\it ($Y$-)list assignment} is a function $L : X \rightarrow [Y]^{< \infty}$. Then $D$ is {\it $L$-(list-)dicolorable} if there is a dicoloring $c$ of $D$ such that $c(x) \in L(x)$ for each $x \in X$. We say that $D$ is {\it degree-list-dicolorable} if, for each list assignment $L$ such that $\vert L(x) \vert \geq d^{\max}(x)$ for all $x \in X$, $D$ is $L$-dicolorable.

Given an undirected graph $G$, we may replace each edge of $G$ with a digon to obtain a directed graph. More precisely, the digraph $D = (X, A)$ is the {\it symmetrization} of the undirected graph $G = (X, E)$ if, for all $x, y \in X$, $\{x, y\} \in E$ if and only if $(x, y) \in A$ and $(y, x) \in A$. Note that, if $D$ is the symmetrization of $G$, then $\overrightarrow{\chi}(D) = \chi(G)$. The digraph $D$ is a {\it symmetric cycle} if $D$ is the symmetrization of an undirected cyclic graph, and $D$ is a {\it complete symmetric digraph} if it is the symmetrization of an undirected complete graph.

Conversely, by forgetting the orientations of the arcs in $D$, we obtain an undirected graph, called the {\it underlying graph} $\tilde{D}$ of $D$. The vertex set of $\tilde{D}$ is $X$, and two vertices $x, y \in X$ are adjacent in $\tilde{D}$ if and only if either $(x, y) \in A$ or $(y, x) \in A$.

Many definitions from the theory of undirected graphs may now be imported to the digraph context. In particular, we say that $D$ is {\it locally finite} if $\tilde{D}$ is locally finite, $D$ is {\it locally countable} if $\tilde{D}$ is locally countable, and $D$ is {\it of bounded degree} if $\tilde{D}$ is of bounded degree. Also, we say that $D$ is {\it connected} if $\tilde{D}$ is connected. If $S \subseteq X$ and $\vert S \vert \geq 2$, then $S$ is {\it biconnected} if the induced sub-digraph $\tilde{D}[S]$ is connected and the induced sub-digraphs $\tilde{D} [S \setminus \{s\}]$ are connected for all $s \in S$. A {\it block} in $D$ is a maximal biconnected set in $D$. A connected component $C$ of $D$ is a {\it Gallai tree} if each block in $D[C]$ induces a dicycle, a (finite) odd symmetric cycle, or a (finite) complete symmetric digraph.

Throughout, we consider descriptive versions of the combinatorial notions for digraphs. For the fundamentals of descriptive set theory, we refer the reader to \cite{kechris1995}.

    \begin{defn}
        Let $X$ be a standard Borel space. A \emph{Borel digraph} on $X$ is a digraph $D = (X, A)$, where $A$ is Borel in the product topology on $X^2$. If $Y$ is a Polish space, then a dicoloring $c : X \rightarrow Y$ of $D$ is a \emph{Borel dicoloring} if $c$ is a Borel function. The \emph{Borel dichromatic number} of $D$, denoted $\overrightarrow{\chi}_B(D)$, is the least $d \in \N$ such that there is a Borel $d$-dicoloring of $D$.

        A \emph{Borel ($Y$-)list assignment} is a Borel function $L : X \rightarrow [Y]^{< \infty}$. We say that $D$ is \emph{Borel degree-list-dicolorable} if, for any Borel list assignment $L$ such that $\vert L(x) \vert \geq d^{\max}(x)$ for all $x \in X$, $D$ is Borel $L$-dicolorable.
    \end{defn}

\section{Definable Digraph Combinatorics}
In this section, we prove some initial results on the definable combinatorics of digraphs. We show first that, if the minimum degree of each vertex in a locally finite Borel digraph is at most $d$, then the Borel dichromatic number is at most $d + 1$. This result is deduced as a corollary of a stronger theorem (Theorem~\ref{thm:degplusone}) on list coloring.

The proof resembles that of Proposition~4.6 in \cite{kst1999}. The idea is as follows: We first separate the vertex set of the underlying graph into countably many independent sets $A_0, A_1, \dots$. Then we list-dicolor the independent sets in order. After we have list-dicolored $A_0 \cup \cdots \cup A_n$, we update the list assignments of each vertex $x \in A_{n + 1}$ by removing the colors that appear among both the out-neighbors and the in-neighbors of $x$.

\begin{thm}
\label{thm:degplusone}
    Let $D$ be a locally finite Borel digraph on a standard Borel space $X$. Let $Y$ be Polish, and let $L : X \to [Y]^{< \infty}$ be a Borel list assignment such that $\vert L(x) \vert > d^{\min}(x)$ for each $x \in X$. Then $D$ has a Borel $L$-dicoloring.
\end{thm}

\begin{proof}
    Since $\tilde{D}$ is locally finite, it follows from \cite[Proposition~4.3]{kst1999} that there is a countable Borel proper coloring $c$ of $\tilde{D}$. For each $n \in \N$, define $A_n = \{x \in X : c(x) = n \}$. Now let $<_Y$ be a Borel linear ordering on $Y$, and define $c_0 : A_0 \to Y$ by letting $c_0(x)$ be the $<_Y$-least element of $L(x)$ for each $x \in A_0$.

    Now, for each $i \in \N$, let $B_i = \bigcup_{j \leq i} A_j$. Fix $n > 0$, and assume that, for each $i \leq n$, a Borel list assignment $L_i : A_i \to [Y]^{< \infty}$ and a Borel function $c_i : B_i \to Y$ have been defined so that the following conditions hold:
    \begin{itemize}
        \item $L_0(x) = L(x)$ for all $x \in A_0$;
        \item $c_i \restriction B_j = c_j$ for all $j < i$;
        \item $L_i(x) \subseteq L(x)$ for all $x \in A_i$;
        \item $L_i(x) \cap \{\alpha \in Y : \exists y, z \in B_{i - 1} \text{ such that } y \in N^+(x), z \in N^-(x), \text{and } \\ c_{i - 1}(y) = c_{i - 1}(z) = \alpha \} = \emptyset$ for all $x \in A_i$; and
        \item $c_i(x) \in L_i(x)$ for each $x \in A_i$.
    \end{itemize}
    We now define a Borel list assignment $L_{n + 1} : A_{n + 1} \to [Y]^{< \infty}$ and a Borel function $c_{n + 1} : B_{n + 1} \to Y$ as follows. First, for each $x \in A_{n + 1}$, let
    $$L_{n + 1}(x) = L(x) \setminus \{\alpha \in Y : \exists y, z \in B_n \text{ such that } y \in N^+(x), z \in N^-(x), \text{and } c_n(y) = c_n(z) = \alpha \}.$$
    Note that the cardinality of the set $\{\alpha \in Y : \exists y, z \in B_n \text{ such that } y \in N^+(x), z \in N^-(x), \text{and } c_n(y) = c_n(z) = \alpha \}$ is at most $d^{\min}(x)$. Since $\vert L(x) \vert > d^{\min}(x)$, we have $\vert L_{n + 1}(x) \vert > 0$. Therefore, for each $x \in B_{n + 1}$, we may define $c_{n + 1}(x) = c_n(x)$ if $x \in B_n$ and $c_{n + 1}(x)$ is the $<_Y$-least color in $L_{n + 1}(x)$ if $x \in A_{n + 1}$.

    Now let $c' = \bigcup_{n \in \N} c_n$. We claim that $c'$ is a Borel $L$-dicoloring of $D$. Assume for contradiction that there is a $c'$-monochromatic dicycle $C = (x_0, x_1, \dots, x_k = x_0)$. Let $n \in \N$ be maximal such that $C \cap A_n \neq \emptyset$, and take $i < k$ with $x_i \in A_n$. Then $c(x_i) > c(x_{i - 1})$ and $c(x_i) > c(x_{i + 1})$. Set $\alpha = c'(x_{i - 1}) = c'(x_{i + 1})$. Then $\alpha \notin L_{c(x_i)}(x_i)$, a contradiction since $c'(x_i) = \alpha$ as well.
\end{proof}

\begin{cor}
    Let $d \in \N$, and let $D$ be a locally countable Borel digraph on a standard Borel space $X$. Suppose $d^{\min}(x) \leq d$ for all $x \in X$. Then $\overrightarrow{\chi}_B(D) \leq d + 1$.
\end{cor}

\begin{proof} 
    In the statement of Theorem~\ref{thm:degplusone}, take $Y = \N$ and $L(x) = \{0, \dots, d\}$ for each $x \in X$.
\end{proof}

We remark that, if the maximum degree of each vertex of $D$ is finite, then $\tilde{D}$ is locally finite. So by Proposition~4.3 of \cite{kst1999}, there is a countable Borel proper coloring of $\tilde{D}$ and hence clearly a countable Borel dicoloring of $D$. This was also observed by Raghavan and Xiao \cite{dx2024}, who prove a dichotomy theorem characterizing which Borel digraphs have countable Borel dicolorings. In particular, Raghavan and Xiao define a canonical digraph -- an analogue of the graph $\mathcal{G}_0$ (\cite[Definition~6.1]{kst1999}) -- and prove that the only digraphs that do not have countable Borel dicolorings are those digraphs that admit continuous homomorphisms from this canonical digraph.

\section{Measurable Brooks's Theorem for Dicolorings}
In this section, we prove Theorems \ref{thm:dibrooks} and \ref{thm:digallai}. First, recall from the introduction that Brooks's theorem for dicolorings does not hold in the Borel context; the simplest counterexample is the symmetrization of a $d$-regular acyclic undirected Borel graph that has no Borel proper $d$-coloring (see Theorem 1.3 in \cite{marks2016}). However, even digon-free digraphs provide counterexamples. The following construction involves the (directed) Schreier graph of the free part of a particular group action.

\begin{prop}
\label{prop:no borel dibrooks}
    Let $d \geq 3$. Then there is a Borel digraph $D$ on a standard Borel space $X$ such that $d^+(x) = d^-(x) = d$ for all $x \in X$, $D$ has neither digons nor odd cycles, and $D$ has no Borel $d$-dicoloring.
\end{prop}

\begin{proof}
    For each $i < d$, let $\Gamma_i = \mathbb{Z}/4\mathbb{Z} = \langle \gamma_i \mathrel{\vert} \gamma_i^4 = 1 \rangle$. Let $X =(2^{\N})^{\Gamma_0 * \cdots * \Gamma_{d-1}}$, and consider the left shift action $\Gamma_0 * \cdots * \Gamma_{d-1} \curvearrowright X$ defined by
    $$(\alpha \cdot x)(\beta) = x(\alpha^{-1}\beta)$$
    for all $\alpha, \beta \in \Gamma_0 * \cdots * \Gamma_{d-1}, x \in X$. Define a Borel digraph $D$ on $\text{Free}(X)$ by directing an arc from $x$ to $y$ if $x \neq y$ and $\gamma_i \cdot x = y$ for some $i < d$. Then it is easy to see that $d^+(x) = d^-(x) = d$ for all $x \in \text{Free}(X)$ and that $D$ has neither digons nor odd cycles.
    
    Assume for contradiction that there is a Borel $d$-dicoloring $c$ of $D$. For each $i < d$, let $A_i = \{x \in X : c(x) = i \}$. Then the sets $A_0, A_1, \dots, A_{d - 1}$ form a Borel partition of $\text{Free}(X)$. By Theorem 1.2 in \cite{marks2017}, there are $i < d$ and a continuous injection $f : \text{Free}((2^{\N})^{\Gamma_i}) \rightarrow \text{Free}(X)$ that is equivariant under the left shift actions of $\Gamma_i$ on $\text{Free}((2^{\N})^{\Gamma_i})$ and $\text{Free}(X)$ with $\text{range}(f) \subseteq A_i$.

    Now let $y \in \text{range}(f)$; then for some $x \in \text{Free}((2^{\N})^{\Gamma_i})$, $f(x) = y$. Notice that $(f(x), \gamma_i \cdot f(x), \gamma_i^2 \cdot f(x), \gamma_i^3 \cdot f(x), f(x))$ is a directed cycle in $D$. However, since $f$ is $\Gamma_i$-equivariant, we have $\gamma_i \cdot f(x) = f(\gamma_i \cdot x)$, $\gamma_i^2 \cdot f(x) = f(\gamma_i^2 \cdot x)$, and $\gamma_i^3 \cdot f(x) = f(\gamma_i^3 \cdot x)$. Then $f(x), \gamma_i \cdot f(x), \gamma_i^2 \cdot f(x), \gamma_i^3 \cdot f(x)$ all belong to $\text{range}(f)$ and hence to $A_i$, contradicting that $A_i$ contains no dicycles.
\end{proof}

\begin{rem}
    The above proposition and its proof were observed to the author by Andrew Marks (see the Acknowledgments).
\end{rem}

Now we turn to the measurable setting. The proofs of Theorems \ref{thm:dibrooks} and \ref{thm:digallai} rely heavily on the one-ended spanning forest technique developed by Conley, Marks, and Tucker-Drob in \cite{cmt2016}. Recall that, if $f$ is a function on a set $X$, then $f$ is {\it one-ended} if there is no sequence $(x_n)_{n \in \N}$ such that, for each $n \in \N$, $f(x_{n + 1}) = x_n$. Note that one-ended functions do not have fixed points.

We first show that, if a digraph $D$ of bounded degree admits a Borel one-ended function, then $D$ is Borel degree-list-dicolorable. The proof combines the proof of Theorem~\ref{thm:degplusone} with the proof of Lemma 3.9 in \cite{cmt2016}: First, given a Borel one-ended function $f$, we separate the graph into layers using the ranks provided by $f$. We will color these layers one at a time. Each individual layer is further stratified according to a Borel coloring $c : X \to \N$ of the underlying graph. So, within a given layer, we first color the points $x$ such that $c(x) = 0$; then we update the list assignments of the remaining points in that layer and move on to the points $x$ such that $c(x) = 1$; and so on. The presence of the one-ended function ensures that we do not run out of usable colors at any stage.

\begin{thm}
\label{thm:one ended implies dicoloring}
    Let $D$ be a locally finite Borel digraph on a standard Borel space $X$, let $B \subseteq X$ be Borel, and let $f : B \rightarrow X$ be a one-ended Borel function whose graph is contained in $\tilde{D}$. Let $Y$ be Polish, and let $L : X \rightarrow [Y]^{< \infty}$ be a Borel list assignment such that $\vert L(x) \vert \geq d^{\max}(x)$ for all $x \in B$. Then $D[B]$ has a Borel $L$-dicoloring. In particular, $D[B]$ is Borel degree-list-dicolorable.
\end{thm}

\begin{proof}
    For each $n \in \N$, write $f^n[B] = \{x \in X : \text{there exist } x_1, x_2, \dots, x_n \in B \text{ such that } f(x_1) = x \text{ and } f(x_{i + 1}) = x_i \text{ for all $i < n$} \}$. Let $B_n = B \cap (f^n[B] \setminus f^{n + 1}[B])$. Note that, if $n \neq m$, then $B_n \cap B_m = \emptyset$. We claim that $B = \bigcup_{n \in \N} B_n$. Assume for contradiction that there is $x \in B \setminus \bigcup_{n \in \N} B_n$. Then $x \in f^n[B]$ for all $n \in \N$. So the set $f^{-\N}(x) = \{y \in B : \text{for some $n \geq 1$, there are } \\ x_1, x_2, \dots, x_{n - 1} \in B \text{ such that } f(x_1) = x, f(x_{i + 1}) = x_i \text{ for all $i < n - 1$, } f(y) = x_{n -1} \}$ is an infinite, finitely branching tree with root $x$. By K{\" o}nig's lemma, there is an infinite branch $(y_n)_{n \in \N}$ through $f^{-\N}(x)$. Then for each $n \in \N$, $f(y_{n + 1}) = y_n$, contradicting that $f$ is one-ended.

    As in the proof of Theorem~\ref{thm:degplusone}, we partition $X$ into countably many Borel $\tilde{D}$-independent sets $A_0, A_1, \dots$. We also fix a Borel linear ordering $<_Y$ of $Y$. Now we proceed to $L$-dicolor $D[B]$ one layer at a time. First, we work in $B_0$. For each $i \in \N$, define $B_0^i = B_0 \cap A_i$, and define $B_0^{\leq i} = \bigcup_{j \leq i} B_0^j, B_0^{< i} = \bigcup_{j < i} B_0^j$. We first color the points in $B_0^0$, then the points in $B_0^1$, and so on.

    We now define a Borel function $c_0^0 : B_0^0 \to Y$. Let $x \in B_0^0$. Then $x$ has at least one neighbor, namely $f(x)$. So since $\vert L(x) \vert \geq d^{\max}_D(x)$, it follows that $L(x) \neq \emptyset$, and so we may define $c_0^0(x)$ to be the $<_Y$-least element of $L(x)$.

    Now fix $n \in \N$, and suppose that for each $i \leq n$, a Borel function $c_0^i : B_0^{\leq i} \to Y$ has been defined so that:
    \begin{itemize}
        \item If $j < i$, then $c_0^i \restriction B_0^{\leq j} = c_0^j$;
        \item For all $x \in B_0^{\leq i}$, $c_0^i(x) \in L(x)$; and
        \item For all $x \in B_0^{\leq i}$, if $x$ has an out-neighbor $y \in B_0^{< i}$ and an in-neighbor $z \in B_0^{< i}$ with $\alpha = c_0^{i - 1}(y) = c_0^{i - 1}(z)$, then $c_0^i(x) \neq \alpha$.
    \end{itemize}
    Now for each $x \in B_0^{n + 1}$, let
    $$L_0^{n + 1}(x) = L(x) \setminus \{\alpha \in Y : \exists y, z \in B_0^{\leq n} \text{ such that } y \in N^+(x), z \in N^-(x), \text{and } c_0^n(y) = c_0^n(z) = \alpha \}.$$
    Then since $x$ has at least one neighbor, namely $f(x)$, not contained in $B_0$, we have that $\{\alpha \in Y : \exists y, z \in B_0^{\leq n} \text{ such that } y \in N^+(x), z \in N^-(x), \text{and } c_0^n(y) = c_0^n(z) = \alpha \}$ has cardinality at most $d^{\max}_D(x) - 1$, and so
    \begin{align*}
        \vert L_0^{n + 1}(x) \vert & \geq \vert L(x) \vert - (d^{\max}_D(x) - 1) \\
        & \geq d^{\max}_D(x) - (d^{\max}_D(x) - 1) \\
        & = 1.
    \end{align*}
    Therefore, we may extend $c_0^n : B_0^{\leq n} \to Y$ to a function $c_0^{n + 1} : B_0^{\leq (n + 1)} \to Y$ by letting $c_0^{n + 1}(x)$ be the $<_Y$-least element of $L_{n + 1}(x)$ for each $x \in B_{n + 1}$. After this recursive procedure is complete, we define $c_0 = \bigcup_{n \in \N} c_0^n$.

    We now work on $B_1$. As before, we partition $B_1$ into independent sets $B_1^0, B_1^1, \dots$. To define $c_1^0 : B_1^0 \to Y$, we first re-initialize the list assignments by setting
    $$L_1^0(x) = L(x) \setminus \{\alpha \in Y : \exists y, z \in B_0 \text{ such that } y \in N^+(x), z \in N^-(x), \text{and } c_0(y) = c_0(z) = \alpha \}$$
    for each $x \in B_1^0$. Then as before, since for each $x \in B_1^0$ there is a neighbor $f(x)$ of $x$ that does not belong to $B_1$, we have $\vert L_1^0(x) \vert \geq 1$, and so we may define $c_1^0(x)$ to be the $<_Y$-least element of $L_1^0(x)$. Given $n \in \N$, having constructed $c_1^0, \dots, c_1^n$, we define $c_1^{n + 1}$ on $B_1^{n + 1}$ as follows: For a fixed $x \in B_1^{n + 1}$, set
    \begin{align*}
        L_1^{n + 1}(x) = L(x) \setminus \{\alpha \in Y : \exists y, z \in (B_0 \cup B_1^{\leq n}) \text{ such that } y \in N^+(x), z \in N^-(x), \\ \text{and } (c_0 \cup c_1^n)(y) = (c_0 \cup c_1^n)(z) = \alpha \}.
    \end{align*}
    Then since $f(x) \notin B_1$, we have $\vert L_1^{n + 1}(x) \vert \geq 1$. Thus, we may define $c_1^{n + 1}(x)$ to be the $<_Y$-least element of $L_1^{n + 1}(x)$. Once this procedure is complete, define $c_1 = \bigcup_{n \in \N} c_1^n$. Continue on in this way to the sets $B_2, B_3, \dots$, constructing functions $c_2, c_3, \dots$.

    Finally, let $c = \bigcup_{n \in \N} c_n$. Then $c : B \to Y$ is Borel. Assume for contradiction that there is a $c$-monochromatic dicycle $C = (x_0, x_1, \dots, x_k = x_0)$. Let $n \in \N$ be maximal with $C \cap B_n \neq \emptyset$, and let $m \in \N$ be maximal with $C \cap B_n^m \neq \emptyset$. Then there is $i < k$ with $x_i \in B_n^m$. It follows that $x_{i - 1}, x_{i + 1} \notin B_n^m$. So either $x_{i - 1} \in B_{n'}$ for some $n' < n$, or $x_{i - 1} \in B_n^{m'}$ for some $m' < m$, and similarly for $x_{i + 1}$. In any case, it follows that $\alpha = c(x_{i - 1}) = c(x_{i + 1})$ is not an element of $L_n^m(x_i)$, a contradiction since $c(x_i) = \alpha$. Thus, $c$ is a Borel $L$-dicoloring of $D[B]$.
\end{proof}

We will use the following proposition of Conley, Marks, and Tucker-Drob to construct one-ended Borel functions. For a graph $G$ on a set $X$ and for $A \subseteq X$, we use the notation $[A]_G$ to denote the set of all points $x \in X$ such that $x$ has a path through $G$ to some point of $A$. For a vertex $x$ of $G$, we write $[x]_G$ for $[\{x\}]_G$. Also, if $D$ is a digraph, then we write $[A]_D$ for $[A]_{\tilde{D}}$.

\begin{prop}
    \cite[Proposition~3.1]{cmt2016}
    \label{prop:one ended exists}
    Let $G$ be a locally finite Borel graph on a standard Borel space $X$, and let $A \subseteq X$ be Borel. Then there is a one-ended Borel function $f : ([A]_G \setminus A) \rightarrow [A]_G$ whose graph is contained in $G$.
\end{prop}

Our next step towards proving Theorem \ref{thm:dibrooks} is to show that digraphs of bounded degree are Borel degree-list-dicolorable on connected components which contain vertices whose minimum degree is smaller than their maximum degree. To construct the dicoloring, we first reserve a set of small-minimum-degree vertices that is independent in the underlying graph and dicolor the vertices outside of this reserved set using a one-ended function. Then, since each reserved vertex has small minimum degree, there is at least one color in its list of available colors which does not appear among both its out-neighbors and its in-neighbors; then the initial dicoloring can be extended to this vertex. The proof resembles the first part of the proof of Theorem 1.2 in \cite{cmt2016}.

\begin{prop}
\label{prop:small degree}
    Let $D$ be a Borel digraph of bounded degree on a standard Borel space $X$, and let $B = \{x \in X : d^{\min}(x) < d^{\max}(x) \}$. Then $D[[B]_D]$ is Borel degree-list-dicolorable.
\end{prop}

\begin{proof}
    Let $Y$ be Polish, and let $L : X \rightarrow [Y]^{< \infty}$ be a list assignment such that $\vert L(x) \vert \geq d^{\max}(x)$ for each $x \in X$. Since $D$ is of bounded degree, the underlying graph $\tilde{D}$ is also of bounded degree. Therefore, by Proposition 4.2 in \cite{kst1999}, there is a Borel maximal $\tilde{D}$-independent set $B' \subseteq B$. Since $[B']_D = [B]_D$, by Proposition \ref{prop:one ended exists}, there is a one-ended Borel function $f : ([B]_D \setminus B') \rightarrow [B]_D$ whose graph is contained in $\tilde{D}$. Then by Theorem \ref{thm:one ended implies dicoloring}, $D[[B]_{D} \setminus B']$ has a Borel $L$-dicoloring.

    Now we extend $c$ to a function $c' : [B]_{D} \rightarrow Y$ by letting $c'(x) = c(x)$ if $x \notin B'$. To handle points in $B'$, first let $<_Y$ be a Borel linear ordering of $Y$. Then for each $x \in B'$, define $c'(x)$ to be the $<_Y$-least color $\alpha \in L(x)$ such that there are no $y \in N^+(x)$ and $z \in N^-(x)$ with $c(y) = c(z) = \alpha$. Such an $\alpha$ exists since $d^{\min}(x) < d^{\max}(x)$, so that at most $d^{\max}(x) - 1$ colors from $L(x)$ appear in $c[N^+(x)] \cap c[N^-(x)]$. Then $c'$ is a Borel $L$-dicoloring of $D[[B]_D]$.
\end{proof}

From now on, we say that $D$ is {\it Eulerian} if, for each vertex, its in-degree and its out-degree are the same. So, the proposition above demonstrates that, if $D$ is a bounded-degree Borel digraph that is not Eulerian, then $D$ is Borel degree-list-dicolorable.

As a corollary to the proof of Proposition \ref{prop:small degree}, we obtain the following.

\begin{cor}
\label{cor:small degree dibrooks}
    Let $D$ be a Borel digraph on a standard Borel space $X$, and let $d$ be such that $d^{\max}(x) \leq d$ for all $x \in X$. Let $B = \{x \in X : d^{\min}(x) < d \}$. Then $D[[B]_D]$ is Borel $d$-dicolorable.
\end{cor}

\begin{proof}
    Take $B'$ and $c$ as in the proof of the previous proposition, and let $L(x) = \{0, 1, \dots, d - 1\}$ for all $x \in X$. Then extend $c$ to a function $c'$ on $X$ by setting $c'(x) = c(x)$ for all $x \notin B'$; for each $x \in B'$, let $c'(x) = \alpha$, where $\alpha$ is the least color such that $\alpha$ does not belong to both $N^+(x)$ and $N^-(x)$. Such an $\alpha$ exists since $d^{\min}(x) < d$.
\end{proof}

Next we work towards proving Theorem \ref{thm:digallai}, which will be used in the proof of Theorem \ref{thm:dibrooks}. In particular, we show that digraphs of bounded degree are Borel degree-list-dicolorable on connected components which are not Gallai trees. The proof is similar to the proof of Theorem 4.1 in \cite{cmt2016}. The argument again involves reserving a set $B''$ of points. These points will belong to biconnected sets that do not induce dicycles, odd symmetric cycles, or complete symmetric digraphs and that are sufficiently well-separated in the underlying graph. Once a degree-list-dicoloring $a$ of the points outside $B''$ has been constructed, the points of $B''$ will be colored in a two-step procedure: First, the values of $a$ are changed on certain neighborhoods of the points in $B''$; then, the dicoloring that results from changing $a$ on these neighborhoods will be extended until all points are colored. The separation between the points of $B''$ will guarantee that the alterations made to $a$ on different regions of the digraph do not interfere with one another.

Suppose $D$ is a directed graph. Let $C = (x_0, x_1, \dots, x_k = x_0)$ be a cycle in $\tilde{D}$. An \emph{orientation} of $C$ is a sequence $(e_0, e_1, \dots, e_k)$ of arcs in $A(D)$ such that, for each $i < k$, the arc $e_i$ is either $(x_i, x_{i + 1})$ or $(x_{i + 1}, x_i)$. The biconnected sets to which the points of $B''$ will belong will be required to have the following form:

\begin{defn}
\label{def:goodcycle}
    Let $D$ be a digraph on a set $X$. Let $G = (x_0, x_1, \dots, x_k = x_0) \subseteq X$ be a cycle in $\tilde{D}$. Then $G$ is a \emph{good cycle} if the following two conditions hold:
    \begin{enumerate}
        \item There is some orientation of $G$ that is not a dicycle; and
        \item $D[G]$ is not an odd symmetric cycle or a complete symmetric digraph.
    \end{enumerate}
\end{defn}

The following proposition can be deduced from the proof of Theorem~2.1 in \cite{hm2011}. It shows that any ``good'' biconnected set of size at least $3$ contains a good cycle.

\begin{prop}
    \cite{hm2011}
    \label{prop:good cycles}
    Let $D$ be an Eulerian digraph on a set $X$, and let $M \subseteq X$ be a finite biconnected set in $D$ such that $\vert M \vert \geq 3$ and $D[M]$ is not a dicycle, an odd symmetric cycle, or a complete symmetric digraph. Then there is a good cycle $G$ contained in $M$.
\end{prop}

\begin{proof}
    Assume first that $\vert M \vert = 3$. Then by assumption, $D[M]$ is not a dicycle. If $D[M]$ is digon-free, then we may take $G = M$ in Definition~\ref{def:goodcycle}. If $D[M]$ has at least one digon, then (1) in Definition~\ref{def:goodcycle} clearly holds for $G = M$, and since $D[M]$ is not $\overleftrightarrow{K}_3$ by assumption, we may again take $G = M$.

    Now assume $\vert M \vert \geq 4$. Let $(x_0, x_1, \dots, x_k = x_0)$ be a $\tilde{D}$-cycle listing all vertices of $M$. We proceed through several cases. First, if $\tilde{D}[M]$ is a cycle, then since $D[M]$ is neither a dicycle nor an odd symmetric cycle, we may take $G = M$ in Definition~\ref{def:goodcycle}.
    
    Otherwise, $\tilde{D}[M]$ is not a cycle, and so there are distinct $i, j < k$ such that $x_i$ and $x_j$ are adjacent in $\tilde{D}$ and $\vert i - j \vert > 1$. Then there are three internally vertex-disjoint $\tilde{D}$-paths $P_0, P_1, P_2$ between $x_i$ and $x_j$ in $M$. For at least two of these paths, their union must be an even cycle in $\tilde{D}$; assume without loss of generality that $C = P_0 \cup P_1$ is an even cycle in $\tilde{D}$. If $D[C]$ is neither a dicycle nor a complete symmetric digraph, then we may take $G = C$. If $D[C]$ is a dicycle, then either $D[P_0 \cup P_2]$ or $D[P_1 \cup P_2]$ is not a dicycle and fails to be symmetric, so that we may take either $G = P_0 \cup P_2$ or $G = P_1 \cup P_2$. 

    Finally, if $D[C]$ is a complete symmetric digraph, then let $K \supseteq C$ be maximal such that $D[K]$ is a complete symmetric digraph, and note that $\vert K \vert \geq 4$. Let $x \in M \setminus K$, and let $\ell < k$ be such that $x = x_{\ell}$. Then since $x_i, x_j, x_{\ell}$ lie on a $\tilde{D}$-cycle, there are $\tilde{D}$-paths $P$ from $x_{\ell}$ to $x_i$ and $Q$ from $x_{\ell}$ to $x_j$ intersecting only at $x_{\ell}$. Let $y \in K$ be distinct from $x_i, x_j$. If $P \cup Q$ has even length, then concatenate $P \cup Q$ with the (undirected) edges $\{x_i, y\}$ and $\{y, x_j\}$ to form a $\tilde{D}$-cycle $C_y$; if $P \cup Q$ has odd length, then let $y' \in K$ be distinct from $x_i, x_j, y$, and concatenate $P \cup Q$ with the (undirected) edges $\{x_i, y\}, \{y, y'\}, \{y', x_j\}$ to form $C_y$. Either way, we obtain an even cycle $C_y$ in $\tilde{D}$ that contains $x_i, x = x_{\ell}, x_j, y$. Since $D[C]$ is complete symmetric, there is some orientation of $C_y$ that is not a dicycle. So if $D[C_y]$ is not complete symmetric, we may take $G = C_y$. Otherwise, in $D$, there are digons between $x = x_{\ell}$ and the vertices $x_i$, $x_j$, and $y$. By repeatedly re-defining $C_y$ to pass through a different $y \in K$ distinct from $x_i$ and $x_j$ each time, either we find a $C_y$ such that $D[C_y]$ is not complete symmetric, in which case we take $G = C_y$, or we deduce that $x$ admits a digon to each vertex of $K$, contradicting maximality of $K$. This completes the proof.
\end{proof}

Consider an Eulerian digraph $D$. Let $L$ be a list assignment such that $\vert L(x) \vert \geq d^{\max}(x)$ for all vertices $x$, and assume that all vertices of $D$ except one vertex $x_0$ have been $L$-dicolored according to a function $c$. If $\vert L(x_0) \vert > d^{\max}(x_0)$, or if $\vert L(x_0) \vert = d^{\max}(x_0)$ and the number of colors in $L(x_0)$ that appear among both the in-neighbors and the out-neighbors of $x_0$ is strictly less than $d^{\max}(x_0)$, then the coloring $c$ may be extended to $x_0$ by defining $c(x_0)$ to be the least color in $L(x_0)$ which does not appear among both the out-neighbors and the in-neighbors of $x_0$. Otherwise, the following proposition of Harutyunyan and Mohar shows that, by uncoloring any neighbor $y$ of $x_0$ and then coloring $x_0$ with the color that $y$ previously had, we obtain a new $L$-dicoloring $c'$ of $D[X \setminus \{y \}]$.

\begin{prop}
    \cite[Lemma~2.2]{hm2011} 
    \label{prop:single vertex extension}
    Let $D$ be an Eulerian digraph on a set $X$, and let $L$ be a list assignment such that $\vert L(x) \vert \geq d^{\max}(x)$ for all $x \in X$. Let $x_0 \in X$, and suppose $c$ is an $L$-dicoloring of $D[X \setminus \{x_0\}]$ such that, for each $\alpha \in L(x_0)$, $x_0$ has both an out-neighbor and an in-neighbor of color $\alpha$. Let $y$ be a neighbor of $x_0$, and define $c'$ on $X \setminus \{y\}$ by $c'(x_0) = c(y)$ and $c'(z) = c(z)$ if $z \neq x_0$. Then $c'$ is an $L$-dicoloring of $D[X \setminus \{y\}]$.
\end{prop}

The final result that we need before proving Theorem \ref{thm:digallai} can be deduced from Lemmas 2.4 and 2.5 in \cite{hm2011}. It shows that, if $G$ is a good cycle and all vertices of $D$ except one vertex in $G$ have been $L$-dicolored, then by repeatedly uncoloring and then coloring pairwise adjacent vertices in $G$ as in the statement of Proposition \ref{prop:single vertex extension}, we reach a vertex $y \in G$ for which some element of $L(y)$ does not appear among both the out-neighbors and the in-neighbors of $y$. At this stage, the $L$-dicoloring may be extended to $y$.

\begin{prop}
    \cite[Lemma~2.5]{hm2011}
    \label{prop:extending to good cycle}
    Let $D$ be an Eulerian digraph on a set $X$, and let $L$ be a list assignment for $D$ such that $\vert L(x) \vert \geq d^{\max}(x)$ for all $x \in X$. Let $G \subseteq X$ be a good cycle. Let $x \in G$, and suppose that there is an $L$-dicoloring $a$ of $D[X \setminus \{x\}]$. Then there are a sequence $(x = x_0, x_1, \dots, x_k)$ of vertices in $G$ and a sequence $(a = a_0, a_1, \dots, a_k)$ of functions such that, for each $i < k$:
    \begin{enumerate}
        \item $x_i$ is $\tilde{D}$-adjacent to $x_{i + 1}$;
        \item $a_i$ is an $L$-dicoloring of $D[X \setminus \{x_i\}]$, and $a_k$ is an $L$-dicoloring of $D[X \setminus \{x_k\}]$;
        \item For each $\alpha \in L(x_i)$, $\alpha \in a_i[N^-(x_i)] \cap a_i[N^+(x_i)]$;
        \item $a_{i + 1}$ is obtained from $a_i$ by uncoloring $x_{i + 1}$ and coloring $x_i$ with $a_i(x_{i + 1})$ (i.e., through an application of Proposition \ref{prop:single vertex extension}); and
        \item There is some color $\alpha \in L(x_k)$ such that either $\alpha \notin a_k[N^-(x_k)]$ or $\alpha \notin a_k[N^+(x_k)]$.
    \end{enumerate}
\end{prop}

Finally, recall that the {\it boundary} of a set $S$ in the digraph $D = (X, A)$, denoted $\partial S$, is the set $\partial S = \{x \in X : x \notin S \text{ but there is $y \in S$ with $x \mathrel{\tilde{D}} y$} \}$.

Now we prove Theorem \ref{thm:digallai}. We restate the theorem here for convenience.

\begin{thm}
\label{thm:no gallai}
    Let $D$ be a Borel digraph of bounded degree on a standard Borel space $X$. Let $L$ be a Borel list assignment for $D$ such that $\vert L(x) \vert \geq d^{\max}(x)$ for all $x \in X$, and let $B = \{x \in X : D[[x]_D] \text{ is not a Gallai tree} \}$. Then $D[B]$ has a Borel $L$-dicoloring. In particular, $D[B]$ is Borel degree-list-dicolorable.
\end{thm}

\begin{proof}
    Note that, by Theorem \ref{prop:small degree}, we may assume that $d^{\max}(x) = d^{\min}(x)$ for all $x \in X$, that is, that $D$ is Eulerian.

    Throughout the proof, fix $d \in \N$ such that $d^{\max}(x) \leq d$ for all $x \in X$. Also, we shall call a biconnected set $S$ {\it bad} if $D[S]$ is a dicycle, an odd symmetric cycle, or a complete symmetric digraph. For technical reasons, we assume first that each block of $D[B]$ that is not bad has cardinality at least $3$. We shall explain at the end of the proof that there is a Borel $L$-dicoloring of the connected components of $D[B]$ that have non-digon blocks of cardinality $2$.

    Before we continue, we first need the following technical claim.

    \begin{claim}
        Let $x \in X$. Then $D[[x]_D]$ is not a Gallai tree if and only if there is a finite connected set $T \subseteq [x]_D$ such that the induced sub-digraph $D[T]$ contains a block\footnote{When we consider blocks in $D[T]$, we mean sets that are blocks relative to $D[T]$, not necessarily all of $D$.} of cardinality at least $3$ that is not bad.
    \end{claim}

    \begin{proof}[Proof of Claim.]
    ($\Rightarrow$) Suppose $D[[x]_D]$ is not a Gallai tree. Then there is a block in $[x]_D$ that does not induce a dicycle, an odd symmetric cycle, or a complete symmetric digraph. If there is a finite such block $M$ in $[x]_D$, then since $M$ has cardinality at least $3$ by the assumption at the beginning of the proof, we may take $T = M$ in the statement of the Claim.
    
    Otherwise, there is an infinite block $M$ in $[x]_D$. Now let $y, z \in M$ be such that $d(y, z) \geq d + 2$; such $y, z$ exist because $M$ is an infinite connected subset of a locally finite graph. Since $M$ is biconnected, there are two internally-vertex-disjoint paths $P_0, P_1$ from $y$ to $z$. Note that each of $P_0, P_1$ has length at least $d + 2$. Again since the $(d+ 2)$-neighborhood of $P_0 \cup P_1$ is a finite set in a locally finite graph, there is some $v \in M$ such that $d(v, P_0 \cup P_1) \geq d + 2$. Since $M$ is connected, there is a path $Q = (y = q_0, q_1, \dots, q_k = v)$ through $M$ from $y$ to $v$.

    {\it Case 1:} $Q$ contains some point of $P_0 \cup P_1$ besides $y$. Let $i < k$ be maximal such that $q_i \neq y$ and $q_i \in P_0 \cup P_1$, and write $y' = q_i$. Since $M$ is biconnected, there is a path $Q' = (y = r_0, r_1, \dots, r_l = v)$ through $M$ from $y$ to $v$ that is internally-vertex-disjoint from $Q$. Let $j < l$ be maximal such that $r_j \in P_0 \cup P_1$, and write $y'' = r_j$. Then $y', y''$ are distinct points of $P_0 \cup P_1$, and there are three internally-vertex-disjoint paths from $y'$ to $y''$: These are the two paths $S_0, S_1$ arising from the fact that $y', y''$ are distinct points of the cycle $P_0 \cup P_1$, together with the concatenation $S_2$ of the sub-path of $Q$ from $y'$ to $v$ and the sub-path of $Q'$ from $v$ to $y''$. Notice that the length of $S_2$ and the length of at least one of $S_0, S_1$ are greater than or equal to $d + 2$. So, two of the three paths $S_0, S_1, S_2$ form a cycle $C$ of even length with cardinality at least $d + 2$. Since each vertex of $D$ has maximum degree at most $d$, the induced sub-digraph $D[C]$ is not a complete symmetric digraph. So, if $D[C]$ is not a dicycle, then we may take $T = C$ in the statement of the Claim. If $D[C]$ is a dicycle, then we may take $T = S_0 \cup S_1 \cup S_2$.

    {\it Case 2:} The only point of $P_0 \cup P_1$ in $Q$ is $y$. Again since $M$ is connected, there is a path $R$ through $M$ from $z$ to $v$. If $R$ contains some point of $P_0 \cup P_1$ besides $z$, then we may proceed as in Case 1 by replacing $y$ with $z$ and $Q$ with $R$. Otherwise, there are three internally-vertex-disjoint paths $P_0, P_1$, and $R \cup Q$ (with repetitions removed) between $y$ and $z$. Note that each of these paths has length at least $d + 2$. Then we may again proceed as in Case 1 by replacing $S_0, S_1, S_2$ with $P_0, P_1, R \cup Q$, respectively.
    
    ($\Leftarrow$) Suppose there is a finite connected set $T \subseteq [x]_D$ such that $D[T]$ contains a block $M$ of size at least $3$ that is not bad. If there is an infinite block in $[x]_D$ which contains $M$, then clearly $D[[x]_D]$ is not a Gallai tree, and the proof is complete. If the only blocks in $[x]_D$ which contain $M$ are finite, then since $M$ has cardinality at least $3$, there is no block in $D$ containing $M$ that induces a dicycle, an odd symmetric cycle, or a complete symmetric digraph. So $D[[x]_D]$ is again not a Gallai tree.
    \end{proof}
    
    Now fix a Borel linear ordering $<_X$ on $X$. We write $[E_D]^{< \infty}$ for the standard Borel space consisting of finite sub-digraphs of $D$ that are contained within a single $D$-component. Let $\Phi \subseteq [E_D]^{< \infty}$ be the Borel set consisting of all finite sets $M \subseteq B$ such that $\vert M \vert \geq 3$ and the induced digraph $D[M]$ is biconnected and not bad. We may assume without loss of generality that, if $M, N \in \Phi$ are distinct, then $M \cup \partial M \neq N \cup \partial N$. Otherwise, for each $M \in \Phi$, consider the set $\mathcal{S}$ of $N \in \Phi$ such that $M \cup \partial M = N \cup \partial N$. Then $\mathcal{S}$ is finite. Therefore, by ordering the elements of each $N \in \mathcal{S}$ via $<_X$, we may take the $<_X$-lexicographically-least element of $\mathcal{S}$ to keep in $\Phi$ and remove all other elements of $\mathcal{S}$ from $\Phi$. Note that $\bigcup \Phi$ meets each component of $D[B]$ by the Claim. 
    
    Let $G_I$ be the intersection graph on $[E_D]^{< \infty}$, so that $S, T \in [E_D]^{<\infty}$ are $G_I$-adjacent if and only if $S \neq T$ and $S \cap T \neq \emptyset$. By Proposition~3 of \cite{cm2016}, there is a countable Borel proper coloring $c_I$ of $G_I$. Let now 
    $$\Phi' = \{M \in \Phi : c_I(M \cup \partial M) \leq c_I(N \cup \partial N) \text{ for all $N \in \Phi$ in the same $D[B]$-component as $M$}\}.$$
    Put $B' = \bigcup \Phi'$. Since $\bigcup \Phi$ meets each component of $D[B]$, we have that $B'$ meets each component of $D[B]$. We remark that, if $M, N \in \Phi'$ are distinct and belong to the same component of $D[B]$, then since $M \cup \partial M$ and $N \cup \partial N$ are distinct by the assumption on $\Phi$, and then since $c_I(M \cup \partial M) = c_I(N \cup \partial N)$, we have $(M \cup \partial M) \cap (N \cup \partial N) = \emptyset$.
    
    Take now $M \in \Phi'$. By Proposition~\ref{prop:good cycles}, there is a good cycle in $M$; let $G_M$ be the $<_X$-lexicographically-least such cycle, and let $x_M$ be the $<_X$-least element of $G_M$. Finally, set $B'' = \{x \in X : \text{there is $M \in \Phi'$ such that } x = x_M \}$. Then $B''$ is Borel, and $[B'']_{D} = [B']_{D} = B$.

    So, by Proposition~\ref{prop:one ended exists}, there is a one-ended Borel function from $(B \setminus B'')$ to $B$. Then by Theorem~\ref{thm:one ended implies dicoloring}, there is a Borel $L$-dicoloring $a$ of $D[B \setminus B'']$. 
    
    We proceed to define a Borel $L$-dicoloring $a'$ of $D[B]$. For each $x \in B''$, let $M \in \Phi'$ witness that $x \in B''$, and let $(x = x_0, x_1, \dots, x_k), (a = a_0, a_1, \dots, a_k)$ be as in Proposition~\ref{prop:extending to good cycle}; this proposition may be applied since $G_M$ is a good cycle and $x = x_M \in G_M$. Then the $L$-dicoloring $a_k$ may be extended to $x_k$, since there is some color $\alpha \in L(x_k)$ such that either $\alpha \notin a_k[N^-(x_k)]$ or $\alpha \notin a_k[N^+(x_k)]$; write $a_M$ for this extension of $a_k$. Then define $a' : B \rightarrow Y$ by $a'(x) = a_M(x)$ if $x \in G_M$ and $a'(x) = a(x)$ if $x \notin G_N$ for any $N \in \Phi'$. 
    
    We claim that $a'$ is a dicoloring of $D[B]$. It is enough to show the following.

    \begin{claim}
        Let $x, y \in B''$ be distinct. Let $c$ be a dicoloring of $D[X \setminus \{x, y \}]$ such that, for each $\alpha \in L(x)$, $\alpha \in c[N^-(x)] \cap c[N^+(x)]$, and for each $\beta \in L(y)$, $\beta \in c[N^-(y)] \cap c[N^+(y)]$. Let $x'$ be a neighbor of $x$, and let $y'$ be a neighbor of $y$. Then the function $c'$ defined on $X \setminus \{x', y'\}$ by $c'(z) = c(z)$ if $z \neq x, y$, $c'(x) = c(x')$, and $c'(y) = c(y')$ is an $L$-dicoloring of $D[X \setminus \{x', y'\}]$.
    \end{claim}

    \begin{proof}[Proof of Claim.]
        The claim is clear from Proposition~\ref{prop:single vertex extension} if $x, y$ belong to different components of $D[B]$. Otherwise, let $M, N \in \Phi'$ witness respectively that $x, y$ belong to $B''$. Note that $M \neq N$, so that $(M \cup \partial M) \cap (N \cup \partial N) = \emptyset$. If there is a $c'$-monochromatic dicycle $K$, then $K$ cannot contain $x'$ or $y'$, and $K$ must contain either $x$ or $y$ since $c$ is a dicoloring. So assume without loss of generality that $x \in K$. Then $c'(x) = c(x')$. For each $\alpha \in L(x)$, we have $\alpha \in c[N^-(x)] \cap c[N^+(x)]$, and since $(G_M \cup \partial G_M) \cap (G_N \cup \partial G_N) = \emptyset$ so that $y, y'$ are not neighbors of $x$, it follows that there is exactly one neighbor $z$ of $x$ such that $c'(x) = c'(z)$. So no dicycle containing $x$ is $c'$-monochromatic.
    \end{proof}
    
    This completes the proof in the case where each block of $D[B]$ that is not bad has cardinality at least $3$. For the other case, let 
    
    $$B_0 = \{x \in X : [x]_D \text{ has a non-digon block of cardinality $2$} \}.$$
    
    As before, take a Borel set $\Phi'$ of non-digon blocks of cardinality $2$ in $B_0$ such that, if $M, N$ are distinct elements of $\Phi'$, then $(M \cup \partial M) \cap (M' \cup \partial M') = \emptyset$ and such that $\bigcup \Phi'$ meets each connected component of $D[B_0]$. Then set $B' = \bigcup \Phi'$. For each $M \in \Phi'$, let $x_M$ be the $<_X$-least element of $M$, and collect these points $x_M$ into a set $B''$. Then use Theorem \ref{thm:one ended implies dicoloring} to produce a Borel $L$-dicoloring $a$ of $D[B_0 \setminus B'']$. To color the points of $B''$, note that, if $x \in B''$, then there is some neighbor $y$ of $x$ such that $\{x, y\}$ is a maximal biconnected set that is not a digon. In particular, no dicycle contains both $x$ and $y$. So, if $L(x)$ contains $a(y)$, then $a$ may be extended to $x$ by setting $a(x) = a(y)$. Otherwise, there is some color in $L(x)$ which does not appear among both the out-neighbors and the in-neighbors of $x$, and so again $a$ can be extended to $x$ by setting $a(x)$ to be the first such color.
\end{proof}

Now we have nearly all the tools required to prove Theorem \ref{thm:dibrooks}. We next need a result of \cite{cmt2016} which shows that, modulo a null set or a meager set, an undirected acyclic graph in which no connected component has $0$ or $2$ ends admits a one-ended Borel function. Recall that, if $G$ is a locally finite graph on a set $X$, then a set $B \subseteq X$ is {\it $G$-invariant} if, whenever $x \in B$ and there is a path from $x$ to $y$ through $G$, then $y \in B$. If $D$ is a digraph on $X$, then we say $B \subseteq X$ is {\it $D$-invariant} if $B$ is $\tilde{D}$-invariant. A {\it ray} in $G$ is an infinite sequence $(x_n)_{n \in \N}$ of pairwise-adjacent vertices in $G$ such that $x_n \neq x_m$ whenever $n \neq m$. Two rays $r_0, r_1$ in $G$ are {\it end-equivalent} if, whenever $S \subseteq X$ is finite, $r_0$ and $r_1$ eventually lie in the same connected component of $G[X \setminus S]$. End-equivalence is an equivalence relation on the set of rays; the equivalence classes are called {\it ends}.

\begin{thm}
    \cite[Theorem~1.5]{cmt2016} 
    \label{thm:one ended exists mod small}
    Let $G$ be a locally finite acyclic graph on a standard Borel space $X$. Assume no connected component of $G$ has either $0$ or $2$ ends.
    \begin{enumerate}
        \item For any Borel probability measure $\mu$ on $X$, there are a $\mu$-conull, $G$-invariant Borel set $B$ and a one-ended Borel function $f : B \rightarrow X$ whose graph is contained in $G$.
        \item For any Polish topology $\tau$ compatible with the Borel structure on $X$, there are a $\tau$-comeager, $G$-invariant Borel set $B$ and a one-ended Borel function $f : B \rightarrow X$ whose graph is contained in $G$.
    \end{enumerate}
\end{thm}

Next, we show that any bounded-degree Borel digraph $D$ having no connected components that are $0$-ended or $2$-ended Gallai trees is Borel degree-list-dicolorable modulo a small set. For the proof, we apply Theorem \ref{thm:one ended exists mod small} to an acyclic sub-digraph $D'$ of $D$. The digraph $D'$ is constructed by removing edges from the blocks of $\tilde{D}$. Each such block induces either a cycle or a complete graph; from each complete graph, edges are removed until only a maximal-length cycle remains, and then from each cycle, the least (according to some Borel linear ordering on the set of edges) edge is removed.

\begin{thm}
\label{thm:one ended gallai trees}
    Let $D$ be a Borel digraph of bounded degree on a standard Borel space $X$. Assume that $D$ has no finite connected components that are Gallai trees and no infinite connected components that are $2$-ended Gallai trees. Then:
    \begin{enumerate}
        \item For any Borel probability measure $\mu$ on $X$, there is a $\mu$-conull, $D$-invariant Borel set $B$ such that $D[B]$ is Borel degree-list-dicolorable.
        \item For any Polish topology $\tau$ compatible with the Borel structure on $X$, there is a $\tau$-comeager, $D$-invariant Borel set $B$ such that $D[B]$ is Borel degree-list-dicolorable.
    \end{enumerate}
\end{thm}

\begin{proof}
    We prove (1); the proof of (2) proceeds in the same way.

    Let $L$ be a list assignment such that $\vert L(x) \vert \geq d^{\max}(x)$ for all $x \in X$. By the assumptions in the theorem statement together with Theorem \ref{thm:no gallai}, we may assume without loss of generality that each connected component of $D$ is an infinite Gallai tree that does not have $2$ ends. Then each block in $D$ induces a finite dicycle, a finite odd symmetric cycle, or a finite complete symmetric digraph. So, in the underlying graph $\tilde{D}$, each block induces either a cyclic graph or a complete graph.

    Now let $E$ denote the set of edges of $\tilde{D}$. Note that each element of $E$ is contained in a unique block. We now create a subset $E' \subseteq E$ as follows. Let $<_X$ be a Borel linear ordering of $X$, and let $<_E$ be a Borel linear ordering of $E$. From each block of $\tilde{D}$ that induces a cyclic graph, remove the $<_E$-least edge; from each block $M$ of $\tilde{D}$ that induces a complete graph of cardinality at least $3$, write $M = \{x_0, x_1, \dots, x_k\}$ in ascending order according to $<_X$, and then delete all edges of $M$ except $\{x_0, x_1\}, \{x_1, x_2\}, \dots, \{x_k, x_0\}$ before also deleting the $<_E$-least edge in the list $\{x_0, x_1\}, \{x_1, x_2\}, \dots, \{x_k, x_0\}$. Now define a sub-digraph $D'$ of $D$ whose vertex set is $X$ and whose arc set is $A'$, where $(x, y) \in A'$ if and only if $(x, y) \in A$ and $\{x, y\} \in E'$.

    We prove several claims about $D'$. Note first that, if $x, y$ are adjacent in $D$, then there is a path from $x$ to $y$ through $\widetilde{D'}$; indeed, the edge between $x$ and $y$ in the underlying graph $\tilde{D}$ determines a unique block $M$. If $M$ has cardinality $2$, then the edge $\{x, y\}$ belongs to $E'$. Otherwise, if $\tilde{D}[M]$ is a cycle, then there is a path through $M$ from $x$ to $y$ which does not contain the edge $\{x, y\}$; if $\tilde{D}[M]$ is a complete graph, then $\widetilde{D'}[M]$ is a cycle missing just one edge, so that again there is a path through $M$ from $x$ to $y$ which does not contain the edge $\{x, y\}$. This implies that no connected component of $D'$ is $0$-ended, since no connected component of $D$ is $0$-ended. This also implies that, if $B \subseteq X$ is $D'$-invariant, then $B$ is also $D$-invariant.

    Next, to see that $D'$ is acyclic, note that any cycle $C$ in $\tilde{D}$ induces either a cyclic graph or a complete graph. If $C$ induces a cyclic graph, then some edge of $C$ is deleted in passing from $D$ to $D'$. If $C$ induces a complete graph, then let $M$ be the maximal complete graph containing $C$. Then at least one of the edges of $M$ that is removed in passing to $D'$ is an edge of $C$.

    Finally, to see that no connected component of $D'$ is $2$-ended, note first that the number of ends in a connected component of $D$ is less than or equal to the number of ends in the corresponding connected component of $D'$. Suppose now that there are two rays $r, r'$ in the same connected component of $D'$ such that, for some finite set $S \subseteq X$, $r \setminus S$ and $r' \setminus S$ do not eventually lie in the same connected component of $D'[X \setminus S]$. Let $T = S \cup \bigcup\{M \in [E_D]^{<\infty} : M \text{ is a block containing a point of $S$}\}$. Then it is easy to prove that $r \setminus T$ and $r' \setminus T$ do not eventually lie in the same connected component of $D[X \setminus T]$. So, since no connected component of $D$ is $2$-ended, no connected component of $D'$ is $2$-ended.

    It now follows from Theorem \ref{thm:one ended exists mod small} that there are a $\mu$-conull, $D'$-invariant Borel set $B$ and a one-ended Borel function $f : B \rightarrow X$ whose graph is contained in $\widetilde{D'}$. Note that $B$ is $D$-invariant by the claim above and that the graph of $f$ is contained in $\tilde{D}$. By Theorem \ref{thm:one ended implies dicoloring}, $D[B]$ has a Borel $L$-dicoloring.
\end{proof}

Finally, we prove Theorem \ref{thm:dibrooks}. We restate the theorem here for convenience.

\begin{thm}
    Let $D$ be a Borel digraph on a standard Borel space $X$. Suppose there is $d \geq 3$ such that $d^{\max}(x) \leq d$ for all $x \in X$, and assume $D$ does not contain the complete symmetric digraph on $d + 1$ vertices. Then:
    \begin{enumerate}
        \item For any Borel probability measure $\mu$ on $X$, there is a $\mu$-measurable $d$-dicoloring of $D$.
        \item For any Polish topology $\tau$ compatible with the Borel structure on $X$, there is a $\tau$-Baire-measurable $d$-dicoloring of $D$.
    \end{enumerate}
\end{thm}

\begin{proof}
    By Corollary \ref{cor:small degree dibrooks}, we may assume that $d^{\max}(x) = d^{\min}(x) = d$ for all $x \in X$. We show that Theorem~\ref{thm:one ended gallai trees} applies. 
    
    First, assume for contradiction that $D$ has a finite connected component $T$ that is a (directed) Gallai tree. Let $G$ be the block-cut graph of $\tilde{T}$; this is the (undirected) graph whose vertex set is the disjoint union of the set of blocks in $\tilde{T}$ and the set of cut vertices in $\tilde{T}$ and such that there is an edge between a cut vertex and a block if and only if the cut vertex is an element of the block. 
    
    If $\tilde{T}$ is itself not a block, then it is easy to check that $G$ is a nontrivial tree and hence has a leaf. Let $B$ be a leaf in $G$, and note that $B$ must be a block in $\tilde{T}$. Note also that $\tilde{T}[B]$ is regular since $T$ is a Gallai tree. Let $x \in B$ be the (unique) cut vertex of $\tilde{T}$ belonging to $B$. There are two cases.

    {\it Case 1:} The digraph $T[B]$ is symmetric. Note that $x$ has at least one $\tilde{T}$-neighbor not in $B$. Then since $T[B]$ is symmetric, $x$ must have both an out-neighbor not in $B$ and an in-neighbor not in $B$. Hence the degree of $x$ in $\tilde{T}[B]$ is at most $d - 1$. Then if $y \in B$ is distinct from $x$, it follows from the fact that $B$ is a leaf of $G$ that $y$ has no neighbors outside of $B$. So by regularity of $\tilde{T}[B]$, it follows that $\deg_{\tilde{T}}(y) = \deg_{\tilde{T}[B]}(x) \leq d - 1$, contradicting that $d^+(y) = d^-(y) = d$.

    {\it Case 2:} The digraph $T[B]$ is a dicycle. Then once again, if $y \in B$ is distinct from $x$, then $y$ has no neighbors outside of $B$, and hence $d^+(y) = d^-(y) = 1 < d$, a contradiction.

    It follows that $G$ is the trivial graph consisting of a single vertex. Hence $T$ itself is a block, and therefore $T = \overleftrightarrow{K}_{d + 1}$, a contradiction.

    Now, assume for contradiction that $D$ has an infinite connected component $T$ that is a $2$-ended (directed) Gallai tree. We claim that $\tilde{T}$ is a $2$-regular bi-infinite line. We show first that each block of $\tilde{T}$ has cardinality $2$. Suppose $B$ is a block in $\tilde{T}$ of cardinality at least $3$. Then $T[B]$ cannot be $\overleftrightarrow{K}_{d + 1}$, as otherwise $T$ would be finite. Since $T$ is a Gallai tree, we have that $T[B]$ is Eulerian and that, if $y, z \in B$ are any vertices of $B$, then $d^+_{T[B]}(y) = d^+_{T[B]}(z)$. Hence, each vertex of $B$ must have at least one $\tilde{T}$-neighbor outside $B$, and all such neighbors are pairwise distinct. So if $G$ is the block-cut graph of $\tilde{T}$, then the degree of $B$ in $G$ is at least $\vert B \vert \geq 3$.

    Notice also that, in this case, there are no leaves in $G$; indeed, if $B'$ were a leaf of $G$, then we could use regularity of $\tilde{T}[B']$ as before to obtain a contradiction. So each vertex in $G$ has degree at least $2$. Now remove the vertex $B$ from $G$. Since $G$ is a tree in which each vertex has degree at least $2$, the resulting graph contains at least three infinite connected components. Then since $G$ is locally finite, it has at least three ends. However, since $T$ is $2$-ended, and all blocks of $T$ are finite, it follows that $G$ is $2$-ended, a contradiction. Thus all blocks of $\tilde{T}$ have cardinality $2$. It follows that $\tilde{T}$ itself is a tree. However, since $d \geq 3$, each vertex in $\tilde{T}$ has degree at least $3$, contradicting that $\tilde{T}$ is $2$-ended.
    
    Therefore, it follows from Theorem \ref{thm:one ended gallai trees} that $D$ has a ($\mu$- or $\tau$-)measurable $d$-dicoloring.
\end{proof}

\section{Future Directions}
Here we outline two potential future research projects.

{\bf Connections among digraph combinatorics, \textsf{LOCAL} algorithms, and the Lovász local lemma.} Bernshteyn's seminal paper \cite{bernshteyn2023} reveals a deep connection between descriptive graph combinatorics and {\it \textsf{LOCAL} coloring algorithms}, which are efficient distributed algorithms for graph coloring. In particular, deterministic \textsf{LOCAL} coloring algorithms can be used to produce Borel colorings, and randomized \textsf{LOCAL} coloring algorithms can be used to produce ($\mu$- or Baire-)measurable colorings \cite{bernshteyn2023}.

The main tool in the proof of the latter result is a measurable version of the {\it symmetric Lovász local lemma}. Let $p < 1$, and consider a set $\mathcal{S}$ of events, each of which has probability at most $p$. Suppose that each event in $\mathcal{S}$ depends on only a small number $d$ of other events in $\mathcal{S}$. If $p$ and $d$ satisfy a certain relationship, then the local lemma ensures that the probability that none of the events in $\mathcal{S}$ occurs is nonzero. For graph coloring, we may take $\mathcal{S}$ to be the set of events in which some vertex shares a color with one of its neighbors; for a graph of bounded degree, each such event has probability bounded away from $1$. Furthermore, due to the local nature of proper colorings, each event in $\mathcal{S}$ depends on only a small number of other events in $\mathcal{S}$. Bernshteyn's measurable symmetric local lemma then ensures the existence of a measurable coloring in which adjacent vertices do not share colors.

We would like to explore whether there is a digraph version of Bernshteyn's result that efficient randomized \textsf{LOCAL} algorithms yield measurable graph colorings. One strategy would be to attempt to apply the measurable symmetric local lemma. However, unlike the problem of producing a proper coloring, the problem of producing a dicoloring is not a local problem; directed cycles can be arbitrarily long. So, we ask the following question.

\begin{quest}
    Are there efficient algorithms suitable for producing Borel or measurable dicolorings?
\end{quest}

While the symmetric local lemma cannot be directly applied to obtain dicolorings, there are other versions of the local lemma that may be of use. For instance, generalizations of the Lovász local lemma that appear in the book of Alon and Spencer \cite{as2008} allow for dependencies of events in $\mathcal{S}$ on arbitrarily many other elements of $\mathcal{S}$ under certain additional constraints. Work of Kun from 2013 that appears in \cite{bfk2022} further extended these generalizations to infinitary settings. Bernshteyn provides an overview of these results in Section 1.2 of \cite{bernshteyn2019}. Versions of the Lovász local lemma have been important in proving combinatorial results about undirected graphs, including the classical theorem of Johansson \cite{johansson1996} which improves the Brooks's theorem upper bound on chromatic numbers of triangle-free undirected graphs. The following question originated with work of Erd\H{o}s and Neumann-Lara \cite{erdos1979}.

\begin{quest}
\label{quest:johansson}
    Is there an analogue of Johansson's theorem for directed graphs?
\end{quest}

{\bf Applications to hypergraph colorings.} A {\it hypergraph} $\mathcal{H}$ consists of a set $X$ of vertices and a set $\mathcal{E} \subseteq \mathcal{P}(X)$ of {\it hyperedges}. A {\it (proper) vertex coloring} of $\mathcal{H}$ is a function $c$ on the set $X$ of vertices of $\mathcal{H}$ such that, for any hyperedge $E \in \mathcal{E}$, there are $x, y \in E$ such that $c(x) \neq c(y)$. 

The study of descriptive hypergraph combinatorics is currently in its infancy. However, note that, for any directed graph $D$ on a set $X$, producing a dicoloring of $D$ is equivalent to producing a vertex coloring of the hypergraph on $X$ whose hyperedges are the directed cycles of $D$. As a result, it may be instructive to apply the techniques of descriptive dicoloring to obtain new results on descriptive hypergraph coloring.

One open problem to which descriptive digraph combinatorics may be applicable is the following, which originated with work of Day and Marks \cite{dm2020}. 

\begin{quest}
\label{quest:hypergraphs}
    Let $\Gamma = (\Z/2\Z)^{* 3} = \langle \alpha, \beta, \gamma \mathrel{\vert} \alpha^2 = \beta^2 = \gamma^2 = 1 \rangle$, and let $\Gamma \curvearrowright 2^{\Gamma}$ be the action of $\Gamma$ on $2^{\Gamma}$ by left shift. Put $X = \text{Free}(2^{\Gamma})$, and let
    \begin{align*}
        A & = \{\dots, \beta^{-1}\alpha^{-1}, \alpha^{-1}, 1, \beta, \alpha \beta, \dots\}, \\
        B & = \{\dots, \gamma^{-1}\beta^{-1}, \beta^{-1}, 1, \gamma, \beta \gamma, \dots\}, \\
        C & = \{\dots, \alpha^{-1}\gamma^{-1}, \gamma^{-1}, 1, \alpha, \gamma \alpha, \dots\}.
    \end{align*}
    Does there exist a Borel set $S \subseteq X$ such that, for each $x \in X$, both $S$ and $X \setminus S$ meet $A \cdot x$, $B \cdot x$, and $C \cdot x$?
\end{quest}

In the question statement, if we take $\mathcal{H}$ to be the hypergraph whose vertex set is $X$ and whose hyperedge set is $\mathcal{E} = \{A \cdot x : x \in X\} \cup \{B \cdot x : x \in X \} \cup \{C \cdot x : x \in X\}$, then we are asking whether $\mathcal{H}$ has a Borel vertex $2$-coloring. 

The question has a positive answer when ``Borel'' is weakened to ``measurable''. One strategy to make the problem more tractable in the Borel context is to include additional restrictions that result in a more ``local'' problem. This may be done by adding, for some appropriately chosen $n \in \N$, the relations $(\alpha \beta)^n = (\beta \gamma)^n = (\alpha \gamma)^n = 1$ to the group $\Gamma$. Consider then the Schreier digraph $D$ whose vertex set is $X = \text{Free}(2^{\Gamma})$ and in which there is an arc from $x \in X$ to $y \in X$ if $x \neq y$ and $\alpha \cdot x = y$, $\beta \cdot x = y$, or $\gamma \cdot x = y$. If $D$ admits a Borel $2$-dicoloring, then since for each $x \in X$ the sets $A \cdot x, B \cdot x,$ and $C \cdot x$ contain dicycles, this would give the desired vertex coloring of the associated hypergraph. 

\begin{quest}
\label{quest:local hypergraphs}
    In the setting of Question \ref{quest:hypergraphs}, is there $n \in \N$ such that the addition of the relations $(\alpha \beta)^n = (\beta \gamma)^n = (\alpha \gamma)^n = 1$ to the group $\Gamma$ results in a Schreier digraph with a Borel $2$-dicoloring?
\end{quest}

\subsection*{Acknowledgments}
We would like to thank Andrew Marks for many insightful conversations about the content of this paper. In particular, Marks suggested the proof of Proposition \ref{prop:no borel dibrooks}, posed Question \ref{quest:local hypergraphs} and drew our attention to Question \ref{quest:johansson}, and illuminated the connections between dicoloring and hypergraph coloring discussed in Section 5. We would also like to thank Louis Esperet for pointing out that Question \ref{quest:johansson} was originally posed by Erd\H{o}s and Neumann-Lara. In addition, we would like to thank Dilip Raghavan and Ming Xiao for helpful conversations about the results in \cite{dx2024}. Finally, we would like to thank the anonymous referee for many insightful comments.

\bibliographystyle{amsalpha}
\bibliography{references}

\end{document}